\documentclass[english,12pt]{article}
\usepackage[T1]{fontenc}
\usepackage[latin9]{inputenc}
\usepackage{geometry}
\usepackage{bm}
\usepackage{amsmath}
\usepackage{amsthm}
\usepackage{amssymb}
\usepackage{mathdots}
\usepackage{setspace}
\usepackage[authoryear]{natbib}
\makeatletter
\newenvironment{lyxlist}[1]
	{\begin{list}{}
		{\settowidth{\labelwidth}{#1}
		 \setlength{\leftmargin}{\labelwidth}
		 \addtolength{\leftmargin}{\labelsep}
		 }}
	{\end{list}}
\theoremstyle{plain}
\newtheorem{lem}{\protect\lemmaname}
\theoremstyle{plain}
\newtheorem{thm}{\protect\theoremname}
\theoremstyle{plain}
\newtheorem{cor}{\protect\corollaryname}

\makeatother

\usepackage{babel}
\providecommand{\corollaryname}{Corollary}
\providecommand{\lemmaname}{Lemma}
\providecommand{\theoremname}{Theorem}

\begin{document}
\title{Asymptotic normality of the time-domain generalized least squares
estimator for linear regression models}
\author{Hien D. Nguyen\thanks{Corresponding author email: h.nguyen5@latrobe.edu.au. $^{1}$Department
of Mathematics and Statistics, La Trobe University, Bundoora Melbourne
3086, Victoria Australia.}$\text{ }{}^{1}$}
\maketitle
\begin{abstract}
In linear models, the generalized least squares (GLS) estimator is
applicable when the structure of the error dependence is known. When
it is unknown, such structure must be approximated and estimated in
a manner that may lead to misspecification. The large-sample analysis
of incorrectly-specified GLS (IGLS) estimators requires careful asymptotic
manipulations. When performing estimation in the frequency domain,
the asymptotic normality of the IGLS estimator, under the so-called
Grenander assumptions, has been proved for a broad class of error
dependence models. Under the same assumptions, asymptotic normality
results for the time-domain IGLS estimator are only available for
a limited class of error structures. We prove that the time-domain
IGLS estimator is asymptotically normal for a general class of dependence
models.
\end{abstract}
\begin{quote}
\end{quote}

\section{\label{sec:Introduction}Introduction}

Let $\left\{ \mathbf{x}_{t}\right\} _{t=1}^{T}$ be a non-stochastic
sequence of vectors, such that $\mathbf{x}_{t}^{\top}=\left(x_{1t},\dots x_{dt}\right)\in\mathbb{R}^{1\times d}$,
where $t\in\left[T\right]=\left\{ 1,\dots,T\right\} $, $d,T\in\mathbb{N}$,
and $\left(\cdot\right)^{\top}$ is the matrix transposition operator.
Let $\left\{ U_{t}\right\} _{t=1}^{T}$ be a sequence of random errors.
We are interested in the sequence $\left\{ Y_{t}\right\} _{t=1}^{T}$
that is generated by the relationship
\begin{equation}
Y_{t}=\mathbf{x}_{t}^{\top}\bm{\beta}+U_{t}\text{,}\label{eq: Individual regression}
\end{equation}
where $\bm{\beta}^{\top}=\left(\beta_{1},\dots,\beta_{p}\right)\in\mathbb{R}^{1\times p}$
is a vector of non-stochastic regression coefficients. We can write
relationship (\ref{eq: Individual regression}) in the matrix form:
\[
\mathbf{y}_{T}=\mathbf{X}_{T}\bm{\beta}+\mathbf{u}_{T}\text{,}
\]
simultaneously for all $t\in\left[T\right]$, where $\mathbf{X}_{T}\in\mathbb{R}^{T\times d}$
has rows $\mathbf{x}_{t}^{\top}$, $\mathbf{y}_{T}^{\top}=\left(Y_{1},\dots,Y_{T}\right)\in\mathbb{R}^{1\times T}$,
and $\mathbf{u}_{T}^{\top}=\left(U_{1},\dots,U_{T}\right)\in\mathbb{R}^{1\times T}$.

Relationships that are described by (\ref{eq: Individual regression})
are generally referred to as multiple linear regression models and
are ubiquitous in the study of engineering, natural science, and social
science phenomena (see e.g., \citealp{Weisberg2005}). For general
treatments of the topic of linear regression modeling, we refer the
interested reader to the the manuscripts of \citet{Gross:2003aa},
\citet{Seber:2003aa}, and \citet{Yan:2009aa}.

In this article, we consider the scenario where the sequence of errors
$\left\{ U_{t}\right\} _{t=1}^{T}$ is a finite segment of the stationary
sequence $\left\{ U_{t}\right\} _{t=-\infty}^{\infty}$, such that
\[
U_{t}=\sum_{i=-\infty}^{\infty}\theta_{i}E_{t-i}\text{,}
\]
for each $t\in\left[T\right]$, where $\left\{ E_{t}\right\} _{t=-\infty}^{\infty}$
is an independent sequence of random variables, such that $\mathbb{E}\left(E_{t}\right)=0$
and $\text{var}\left(E_{t}\right)=\sigma^{2}<\infty$. If the sequence
of coefficients $\left\{ \theta_{i}\right\} _{i=-\infty}^{\infty}$
is known, then one can write the covariance matrix of $\mathbf{u}_{T}$:
$\bm{\Sigma}_{T}=\mathbb{E}\left(\mathbf{u}_{T}\mathbf{u}_{T}^{\top}\right)$
and use it to construct, the so-called generalized least squares (GLS)
estimator
\begin{equation}
\tilde{\bm{\beta}}_{T}\left(\bm{\Sigma}_{T}\right)=\left(\mathbf{X}_{T}^{\top}\text{\ensuremath{\bm{\Sigma}}}_{T}^{-1}\mathbf{X}_{T}\right)^{-1}\mathbf{X}_{T}^{\top}\bm{\Sigma}_{T}^{-1}\mathbf{y}_{T}\text{,}\label{eq: GLS}
\end{equation}
which is known to be the best linear unbiased estimator (BLUE) of
$\bm{\beta}$ (cf. \citealp[Sec. 6.1.3]{Amemiya1985}). Furthermore,
Theorem 1 of \citet[Ch. 9]{Baltagi:2002aa} states that under the
condition that $\mathbf{C}_{u}=\lim_{T\rightarrow\infty}T^{-1}\mathbf{X}_{T}^{\top}\bm{\Sigma}_{T}^{-1}\mathbf{X}_{T}^{\top}$
exists, we also have the fact that $\tilde{\bm{\beta}}_{T}\left(\bm{\Sigma}_{T}\right)$
is asymptotically normal in the sense that
\[
T^{1/2}\left[\tilde{\bm{\beta}}_{T}\left(\bm{\Sigma}_{T}\right)-\bm{\beta}\right]\overset{\mathcal{L}}{\longrightarrow}\text{N}\left(\mathbf{0},\mathbf{C}_{u}\right)\text{,}
\]
where we denote convergence in law by $\overset{\mathcal{L}}{\longrightarrow}$
and $\text{N}\left(\mathbf{\bm{\mu}},\bm{\Sigma}\right)$ denotes
a normal distribution with mean vector $\mathbf{0}$ and covariance
matrix $\bm{\Sigma}$.

In applications, the coefficients $\left\{ \theta_{i}\right\} _{i=-\infty}^{\infty}$
are rarely if ever known. Thus, the DGP of the error sequence $\left\{ U_{t}\right\} _{t=-\infty}^{\infty}$
is also unknown. In order to proceed to make inference, one generally
assumes a hypothetical DGP for $\left\{ U_{t}\right\} _{t=-\infty}^{\infty}$
that is equivalent to that of some error sequence $\left\{ V_{t}\right\} _{t=-\infty}^{\infty}$.

Let $\left\{ V_{t}\right\} _{t=1}^{T}$ be a finite segment of $\left\{ V_{t}\right\} _{t=-\infty}^{\infty}$,
and let $\mathbf{v}_{T}^{\top}=\left(V_{1},\dots,V_{T}\right)\in\mathbb{R}^{1\times T}$.
Furthermore, write the covariance matrix $\mathbf{v}_{T}$ as $\bm{\Lambda}_{T}=\mathbb{E}\left(\mathbf{v}_{T}\mathbf{v}_{T}^{\top}\right)$.
Then, by replacing $\bm{\Sigma}_{T}$ in (\ref{eq: GLS}) by $\bm{\Lambda}_{T}$,
we obtain the so-called incorrectly-specified GLS (IGLS; \citealp{Koreisha2001})
estimator

\begin{equation}
\tilde{\bm{\beta}}_{T}\left(\bm{\Lambda}_{T}\right)=\left(\mathbf{X}_{T}^{\top}\text{\ensuremath{\bm{\Lambda}}}_{T}^{-1}\mathbf{X}_{T}\right)^{-1}\mathbf{X}_{T}^{\top}\bm{\Lambda}_{T}^{-1}\mathbf{y}_{T}\text{.}\label{eq: IGLS-1}
\end{equation}

The finite-sample properties of (\ref{eq: IGLS-1}) were studied comprehensively
in \citet{Koreisha2001} and \citet{Kariya2004}. General asymptotic
results for the IGLS estimator are more difficult to establish, and
thus only a small number of results are available in the literature.
For example, \citet{Rothenberg1984} considered asymptotic normality
of the IGLS estimator, when $\bm{\Lambda}_{T}$ has first-order autoregressive
(AR) form. Some consistency results regarding the IGLS estimator appear
in \citet{Samarov1987} and \citet{Koreisha2001}. To date, the most
general set of asymptotic theorems regarding the IGLS estimator are
those reported in \citet{Amemiya1973}.

Make the so-called Grenander regularity conditions \citep{Grenander1954}:
\begin{lyxlist}{00.00.0000}
\item [{{[}Gren1{]}}] $\lim_{T\rightarrow\infty}s_{iT}^{2}=\infty$ ($i\in\left[p\right]$),
where $s_{iT}^{2}=\sum_{t=1}^{T}x_{it}^{2}$, 
\item [{{[}Gren2{]}}] $\lim_{T\rightarrow\infty}\left(x_{iT}^{2}/s_{iT}^{2}\right)=0$
($i\in\left[p\right]$), 
\item [{{[}Gren3{]}}] $\lim_{T\rightarrow\infty}\left[\left(s_{iT}s_{jT}\right)^{-1}\sum_{t=1}^{T-k}x_{it}x_{j,t+k}\right]=\rho_{ij}\left(k\right)$
exists for each $i,k\in\left[p\right]$, and
\item [{{[}Gren4{]}}] the matrix $\mathbf{R}\left(0\right)$ is non-singular,
where $\mathbf{R}\left(k\right)$ has element $\rho_{ij}\left(k\right)$
in the $i\text{th}$ row and $j\text{th}$ column.
\end{lyxlist}
Furthermore, make the following additional assumptions.
\begin{lyxlist}{00.00.0000}
\item [{{[}Amem1{]}}] The $t\text{th}$ elements of $\left\{ U_{t}\right\} _{t=-\infty}^{\infty}$
have the form
\[
U_{t}=\sum_{i=1}^{\infty}\alpha_{i}U_{t-i}+E_{t}\text{,}
\]
where $\left\{ E_{t}\right\} _{t=-\infty}^{\infty}$ is a sequence
of independent random variables such that $\mathbb{E}\left(E_{t}\right)=0$
and $\text{var}\left(E_{t}\right)=\sigma^{2}<\infty$, and $\left\{ \alpha_{i}\right\} _{i=1}^{\infty}$
is such that $\left|1-\sum_{i=1}^{\infty}\alpha_{i}e^{\iota i\omega}\right|^{2}>0$
and $\sum_{i=1}^{N}\left|\alpha_{N+i}\right|=O\left(c^{N}\right)$,
where $c\in\left[0,1\right)$ and $\omega\in\left[-\pi,\pi\right]$.
\item [{{[}Amem2{]}}] The sequence $\left\{ U_{t}\right\} _{t=-\infty}^{\infty}$
is hypothesized to be equivalent to the stationary autoregressive
(AR) process $\left\{ V_{t}\right\} _{t=-\infty}^{\infty}$, with
$t\text{th}$ term of the form
\[
V_{t}=\sum_{i=1}^{N}\kappa_{i}V_{t-i}+E_{t}\text{,}
\]
where $\left\{ E_{t}\right\} _{t=-\infty}^{\infty}$ is a sequence
of independent random variables such that $\mathbb{E}\left(E_{t}\right)=0$
and $\text{var}\left(E_{t}\right)=\sigma^{2}<\infty$, and $\left\{ \kappa_{i}\right\} _{i=1}^{N}$
are such that the roots of $\sum_{i=1}^{N}\kappa_{i}\zeta^{i}=1$
(with respect to $\zeta$) are all outside of the unit circle.
\end{lyxlist}
Here $\iota=\sqrt{-1}$ denotes the imaginary unit. See \citet[Ch. 5.2]{Amemiya1985}
for more details regarding AR processes.

Let $\mathbf{S}_{T}$ be the diagonal matrix with $i\text{th}$ element
$s_{iT}$, for $i\in\left[p\right]$. Further, let $\mathbf{Z}_{T}=\mathbf{X}_{T}\mathbf{S}_{T}^{-1}$
and define $\mathbf{H}\left(\omega\right)$ to be a Hermitian matrix
function with positive semidefinite increments such that $\mathbf{R}\left(k\right)=\int_{-\pi}^{\pi}e^{\iota k\omega}\text{d}\mathbf{H}\left(\omega\right)$.
Under {[}Gren1{]}--{[}Gren4{]}, {[}Amem1{]}, and {[}Amem2{]}, \citet{Amemiya1973}
proved that

\[
\mathbf{S}_{T}\left[\tilde{\bm{\beta}}_{T}\left(\bm{\Lambda}_{T}\right)-\bm{\beta}\right]\overset{\mathcal{L}}{\longrightarrow}\text{N}\left(\mathbf{0},\mathbf{C}_{v}\right)\text{,}
\]
where

\[
\mathbf{C}_{v}=\lim_{T\rightarrow\infty}\left(\mathbf{Z}_{T}^{\top}\bm{\Lambda}_{T}^{-1}\mathbf{Z}_{T}\right)^{-1}\mathbf{Z}_{T}^{\top}\bm{\Lambda}_{T}^{-1}\bm{\Sigma}_{T}\bm{\Lambda}_{T}^{-1}\mathbf{Z}_{T}\left(\mathbf{Z}_{T}^{\top}\bm{\Lambda}_{T}^{-1}\mathbf{Z}_{T}\right)^{-1}\text{.}
\]

Furthermore, \citet{Amemiya1973} showed that if we denote the spectral
density functions (SDFs) of the processes $\left\{ U_{t}\right\} _{t=-\infty}^{\infty}$
and $\left\{ V_{t}\right\} _{t=-\infty}^{\infty}$ by $f_{u}\left(\omega\right)$
and $f_{v}\left(\omega\right)$, respectively, then we may write $\mathbf{C}_{v}$
in the spectral form

\begin{equation}
\mathbf{C}_{v}=2\pi\left[\int_{-\pi}^{\pi}\frac{1}{f_{v}\left(\omega\right)}\text{d}\mathbf{H}\left(\omega\right)\right]^{-1}\int_{-\pi}^{\pi}\frac{f_{u}\left(\omega\right)}{f_{v}^{2}\left(\omega\right)}\text{d}\mathbf{H}\left(\omega\right)\left[\int_{-\pi}^{\pi}\frac{1}{f_{v}\left(\omega\right)}\text{d}\mathbf{H}\left(\omega\right)\right]^{-1}\text{.}\label{eq: Spectral IGLS Covariance-1}
\end{equation}
It is remarkable that the IGLS is a time-domain estimator that can
be proved to have a covariance matrix with simple spectral form.

Since the SDF of $\left\{ V_{t}\right\} _{t=-\infty}^{\infty}$ can
be written as
\begin{equation}
f_{v}\left(\omega\right)=\left(2\pi\right)^{-1}\sum_{i=-\infty}^{\infty}\eta_{i}e^{-\iota i\omega}\text{,}\label{eq: fv}
\end{equation}
where $\eta_{i}=\text{cov}\left(V_{t},V_{t+i}\right)$ and $\eta_{i}=\eta_{-i}$
for each $i\in\mathbb{Z}$, we can write

\begin{equation}
\mathbb{E}\left(\mathbf{v}_{T}\mathbf{v}_{T}^{\top}\right)=\bm{\Lambda}_{T}=\left[\begin{array}{ccccc}
\eta_{0} & \eta_{1} & \eta_{2} & \cdots & \eta_{T-1}\\
\eta_{1} & \eta_{0} & \eta_{1}\\
\eta_{2} & \eta_{1} & \eta_{0} &  & \vdots\\
\vdots &  &  & \ddots\\
\eta_{T-1} &  & \cdots &  & \eta_{0}
\end{array}\right]\text{.}\label{eq: autocovariance}
\end{equation}
Thus, since the auto-covariance $\bm{\Lambda}_{T}$ is determined
by (\ref{eq: fv}), we may interchange the notation $\tilde{\bm{\beta}}_{T}\left(\bm{\Lambda}_{T}\right)$
with $\tilde{\bm{\beta}}_{T}\left(f_{v}\right)$. We shall refer to
$\tilde{\bm{\beta}}_{T}\left(f_{v}\right)$ as the time-domain IGLS
estimator in order to differentiate it from the frequency-domain IGLS
estimator that is introduced in the sequel.

Let
\begin{equation}
\mathbf{J}_{\mathbf{x}\mathbf{x}}\left(\omega\right)=\frac{1}{2\pi T}\left(\sum_{t=1}^{T}\mathbf{x}_{t}e^{\iota t\omega}\right)\left(\sum_{t=1}^{T}\mathbf{x}_{t}^{\top}e^{-\iota t\omega}\right)\label{eq: Jxx}
\end{equation}
and
\begin{equation}
\mathbf{J}_{\mathbf{x}Y}\left(\omega\right)=\frac{1}{2\pi T}\left(\sum_{t=1}^{T}\mathbf{x}_{t}e^{\iota t\omega}\right)\left(\sum_{t=1}^{T}Y_{t}e^{-\iota t\omega}\right)\label{eq: Jxy}
\end{equation}
to be the periodogram of $\left\{ \mathbf{x}_{t}\right\} _{t=1}^{T}$
and the cross-spectra periodogram between $\left\{ \mathbf{x}_{t}\right\} _{t=1}^{T}$
and $\left\{ Y_{t}\right\} _{t=1}^{T}$, respectively (see, e.g.,
\citealp[Sec. 11.7]{Brockwell2006}). Using (\ref{eq: Jxx}) and (\ref{eq: Jxy}),
\citet{Hannan1973} proposed to estimate $\bm{\beta}$ by the frequency-domain
IGLS estimator
\begin{equation}
\bar{\bm{\beta}}\left(f_{v}\right)=\left[\sum_{t=1}^{T}f_{v}^{-1}\left(\frac{2\pi t}{T}\right)\mathbf{J}_{\mathbf{x}\mathbf{x}}\left(\frac{2\pi t}{T}\right)\right]^{-1}\sum_{t=1}^{T}f_{v}^{-1}\left(\frac{2\pi t}{T}\right)\mathbf{J}_{\mathbf{x}Y}\left(\frac{2\pi t}{T}\right)\text{.}\label{eq: freq do IGLS}
\end{equation}

Let $\mathcal{F}_{t-1}$ is the $\sigma\text{-algebra}$ that is generated
by the sequence $\left\{ E_{t-i}\right\} _{i=1}^{\infty}$ and make
the following assumptions.
\begin{lyxlist}{00.00.0000}
\item [{{[}Hann1{]}}] The random sequence $\left\{ U_{t}\right\} _{t=-\infty}^{\infty}$
has the form $U_{t}=\sum_{i=-\infty}^{\infty}\theta_{i}E_{t-i}$,
satisfying
\[
\mathbb{E}\left(E_{t}|\mathcal{F}_{t-1}\right)=\mathbb{E}\left(E_{t}^{2}-\mathbb{E}\left(E_{t}^{2}\right)|\mathcal{F}_{t-1}\right)=0
\]
almost surely, and $\sum_{i=-\infty}^{\infty}\left|\theta_{i}\right|<\infty$.
\item [{{[}Hann2{]}}] The random sequence $\left\{ E_{t}\right\} _{t=-\infty}^{\infty}$
has distribution $F_{E_{t}}\left(e\right)$, which satisfies
\[
\lim_{\delta\rightarrow\infty}\sup_{t\in\mathbb{Z}}\int_{\left|e\right|>\delta}e^{2}\text{d}F_{E_{t}}\left(e\right)=0\text{.}
\]
\item [{{[}Hann3{]}}] The SDF $f_{v}\left(\omega\right)$ is real, positive,
continuous, and even over $\omega\in\left[-\pi,\pi\right]$.
\end{lyxlist}
Under {[}Gren1{]}--{[}Gren4{]} and {[}Hann1{]}--{[}Hann3{]}, \citet{Hannan1973}
proved that

\begin{equation}
\mathbf{S}_{T}\left[\bar{\bm{\beta}}\left(f_{v}\right)-\bm{\beta}\right]\overset{\mathcal{L}}{\longrightarrow}\text{N}\left(\mathbf{0},\mathbf{C}_{v}\right)\text{,}\label{eq: AN FD}
\end{equation}
where $\mathbf{C}_{v}$ has form (\ref{eq: Spectral IGLS Covariance-1})
(see also \citealp{Robinson1997}). That is, (\ref{eq: IGLS-1}) and
(\ref{eq: freq do IGLS}) have the same asymptotic distributions when
both {[}Amem1{]} and {[}Amem2{]}, or {[}Hann1{]}--{[}Hann3{]} are
satisfied, in addition to {[}Gren1{]}--{[}Gren4{]}.

It is notable, however, that {[}Hann1{]}--{[}Hann3{]} are more general
assumptions that {[}Amem1{]} and {[}Amem2{]}. Thus, an obvious question
to ask is whether the equivalence in asymptotic distributions between
the time-domain estimator (\ref{eq: IGLS-1}) and frequency-domain
estimator (\ref{eq: freq do IGLS}) remains when one replaces {[}Amem1{]}
and {[}Amem2{]} by assumptions that are more general and closer in
spirit to {[}Hann1{]}--{[}Hann3{]}. In this article, we shall provide
an affirmative answer to this question. Before presenting our main
result, we wish to provide a review of the relevant literature.

The asymptotic covariance form (\ref{eq: freq do IGLS}) was used
by \citet{Engle1974} and \citet{Nicholls1977} to explore the efficiencies
of the OLS estimator, the BLUE, and the IGLS, under various choices
of $f_{u}$ and $f_{v}$, when $d=1$. Some finite sample properties
of the frequency-domain estimator were established in \citet{Engle1974a}
and \citet{Engle1976}.

Under {[}Gren1{]}--{[}Gren4{]}, the IGLS asymptotic covariance (\ref{eq: Spectral IGLS Covariance-1})
was obtained via spectral methods in \citet{Kholevo1969}, \citet{Rozanov1969},
\citet{Kholevo1971}, and \citet[Sec. 7.4]{Ibragimov1978}, under
general conditions (see Lemma \ref{Lem: Ibragimov-1} in the Appendix).
In the cited papers, the IGLS was studied under the name of pseudo-best
estimators. Unfortunately, no asymptotic normality result of the desired
kind were established. It is notable, that \citet{Kholevo1971a} obtained
an asymptotic normality result for the continuous-time least squares
problem that is hypothesized to be transferable to the pseudo-best
estimator case. However, no such result was provided, nor a result
regarding the discrete time case.

Hybrid time and frequency-domain IGLS estimators have also been considered,
as well as extensions upon the frequency-domain estimator theme. Examples
of hybrid estimators include \citet{Samarov1987} and \citet{Hambaba1992}. 

Extensions of the results of \citet{Hannan1973} to account for long-range
dependence appear in \citet{Robinson1997a} and \citet{Hidalgo2002}.
A non-linear frequency-domain estimator appears in \citet{Hannan1971}.
A broad generalization of the frequency-domain estimation approach
to semi-parametric and non-parametric modeling is considered by \citet{Robinson1991}.

Closely related to our article is the report of \citet{Aguero2010},
which establishes the asymptotic equivalence between time and frequency-domain
estimators for linear dynamic system identification problems. See
\citet{Hannan2012} regarding linear dynamic systems. 

Using the Cholesky covariance matrix factorization method of \citet{Wu2003},
\citet{Yang2012} constructed an IGLS estimator that is asymptotically
efficient. Furthermore, they obtain an asymptotic normality result,
under the {[}Gren1{]}--{[}Gren4{]}, using a proof technique that
is adapted from those of \citet[Thm. 10.2.7]{Anderson1971} and \citet[Thm. 9.1.2]{Fuller1996}
(see Lemma \ref{Lem: Anderson 1} in the Appendix). A model averaging
method akin to the construction of \citet{Yang2012} was studied in
\citet{Cheng2015}, and a long memory GLS estimator of the same form
was considered by \citet{Ing2016}.

Also related to our article is the work of \citet{Kapetanios2016},
which proposed to extend the results of \citet{Amemiya1973} in a
different direction. Here, the {[}Gren1{]}--{[}Gren4{]} are replaced
by various stochastic assumptions on the sequences $\left\{ \mathbf{x}_{t}\right\} _{t=1}^{T}$
and $\left\{ U_{t}\right\} _{t=1}^{T}$ that make use of mixing and
stochastic approximation concepts, and higher moment bound (see \citealt[Ch. 6]{Potscher1997}
regarding mixing and approximation concepts). Compared to our work,
the work of \citet{Kapetanios2016} can be seen as a complementary
and parallel direction of generalization of the results of \citet{Amemiya1973}.
Whereas we propose to relax {[}Amem1{]} and {[}Amem2{]}, \citet{Kapetanios2016}
replaces {[}Gren1{]}--{[}Gren4{]}, instead.

The remainder of the manuscript proceeds as follows. In Section 2,
we state and prove our main result. Discussions and remarks are provided
in Section 3. Here, we provide results regarding the practical case,
where $f_{v}$ is both hypothesized and estimated from the data. Necessary
lemmas and technical results are presented in the Appendix.

\section{Main result}

We retain all notation from the introduction. Furthermore for matrices
$\mathbf{A}\in\mathbb{R}^{m\times n}$, let
\[
\left\Vert \mathbf{A}\right\Vert _{\text{op}}=\sup\left\{ \left\Vert \mathbf{A}\bm{x}\right\Vert _{2}/\left\Vert \bm{x}\right\Vert _{2}:\bm{x}\in\mathbb{R}^{n}\backslash\left\{ \mathbf{0}\right\} \right\} 
\]
denote the operator norm, and let
\[
\left\Vert \mathbf{A}\right\Vert _{1}=\max_{j\in\left[n\right]}\sum_{i=1}^{m}\left|a_{ij}\right|\text{ and }\left\Vert \mathbf{A}\right\Vert _{\infty}=\max_{i\in\left[m\right]}\sum_{j=1}^{n}\left|a_{ij}\right|\text{,}
\]
denote the $l_{1}$ and $l_{\infty}$ induced norms, respectively.
For vectors $\bm{a}\in\mathbb{R}^{m}$, we denote the Euclidean norm
of $\bm{a}$ by $\left\Vert \bm{a}\right\Vert _{2}$.

Make the following assumptions.
\begin{lyxlist}{00.00.0000}
\item [{{[}Main1{]}}] The $t\text{th}$ element of the error sequence $\left\{ U_{t}\right\} _{t=-\infty}^{\infty}$
has form
\begin{equation}
U_{t}=\sum_{i=-\infty}^{\infty}\theta_{i}E_{t-i}\text{,}\label{eq: power transfer}
\end{equation}
and $\left\{ E_{t}\right\} _{t=-\infty}^{\infty}$ is an independent
sequence, where
\[
\mathbb{E}\left(E_{t}\right)=0\text{, }\text{var}\left(E_{t}\right)=\sigma^{2}<\infty\text{, }\sum_{i=-\infty}^{\infty}\left|\theta_{i}\right|<\infty\text{, and }0<\left|\sum_{i=-\infty}^{\infty}\theta_{i}e^{-\iota i\omega}\right|\text{.}
\]
\item [{{[}Main2{]}}] The random sequence $\left\{ E_{t}\right\} _{t=-\infty}^{\infty}$
has distribution $F_{E_{t}}\left(e\right)$, which satisfies
\[
\lim_{\delta\rightarrow\infty}\sup_{t\in\mathbb{Z}}\int_{\left|e\right|>\delta}e^{2}\text{d}F_{E_{t}}\left(e\right)=0\text{.}
\]
\item [{{[}Main3{]}}] The SDF $f_{v}\left(\omega\right)$ is real, positive,
continuous, and even over $\omega\in\left[-\pi,\pi\right]$.
\item [{{[}Main4{]}}] The covariance expansion (\ref{eq: fv}) of $f_{v}\left(\omega\right)$
satisfies
\[
\sum_{i=1}^{\infty}i\left|\eta_{i}\right|<\infty\text{.}
\]
\end{lyxlist}
\begin{lem}
Under {[}Gren1{]}--{[}Gren4{]} and {[}Main1{]}--{[}Main4{]}, $\mathbf{S}_{T}\text{var}\left[\tilde{\bm{\beta}}_{T}\left(f_{v}\right)\right]\mathbf{S}_{T}$
approaches
\begin{align}
\mathbf{C}_{v} & =\lim_{T\rightarrow\infty}\left(\mathbf{Z}_{T}^{\top}\bm{\Lambda}_{T}^{-1}\mathbf{Z}_{T}\right)^{-1}\mathbf{Z}_{T}^{\top}\bm{\Lambda}_{T}^{-1}\bm{\Sigma}_{T}\bm{\Lambda}_{T}^{-1}\mathbf{Z}_{T}\left(\mathbf{Z}_{T}^{\top}\bm{\Lambda}_{T}^{-1}\mathbf{Z}_{T}\right)^{-1}\label{eq: Covariance Theorem}\\
 & =2\pi\left[\int_{-\pi}^{\pi}\frac{1}{f_{v}\left(\omega\right)}\text{d}\mathbf{H}\left(\omega\right)\right]^{-1}\int_{-\pi}^{\pi}\frac{f_{u}\left(\omega\right)}{f_{v}^{2}\left(\omega\right)}\text{d}\mathbf{H}\left(\omega\right)\left[\int_{-\pi}^{\pi}\frac{1}{f_{v}\left(\omega\right)}\text{d}\mathbf{H}\left(\omega\right)\right]^{-1}\text{,}\nonumber 
\end{align}
as $T\rightarrow\infty$.
\end{lem}
\begin{proof}
Following from \citet{Amemiya1973}, we write
\begin{equation}
\mathbf{S}_{T}\left[\tilde{\bm{\beta}}_{T}\left(f_{v}\right)-\bm{\beta}\right]=\left(\mathbf{Z}_{T}\bm{\Lambda}_{T}^{-1}\mathbf{Z}_{T}\right)^{-1}\mathbf{Z}_{T}^{\top}\bm{\Lambda}_{T}^{-1}\mathbf{u}_{T}\text{,}\label{eq: Z least squares form}
\end{equation}
and let \textbf{$\mathbf{w}_{T}=\mathbf{Z}_{T}^{\top}\bm{\Lambda}^{-1}\mathbf{u}_{T}$}.
Under {[}Gren1{]}--{[}Gren4{]} and {[}Main3{]},
\[
\lim_{T\rightarrow\infty}\mathbf{Z}_{T}^{\top}\bm{\Lambda}_{T}^{-1}\mathbf{Z}_{T}=\frac{1}{2\pi}\int_{-\pi}^{\pi}\frac{1}{f_{v}\left(\omega\right)}\text{d}\mathbf{H}\left(\omega\right)\text{,}
\]
by Lemma \ref{Lem: Anderson 1}. By Lemma \ref{lem: Gray nonsing},
$\bm{\Lambda}_{T}$ is invertible and thus, for any $T$, $\bm{\Lambda}_{T}^{-1}$
exists. Thus, we have
\[
\mathbf{S}_{T}\text{var}\left[\tilde{\bm{\beta}}_{T}\left(f_{v}\right)\right]\mathbf{S}_{T}=\left(\mathbf{Z}_{T}\bm{\Lambda}_{T}^{-1}\mathbf{Z}_{T}\right)^{-1}\mathbf{Z}_{T}^{\top}\bm{\Lambda}_{T}^{-1}\bm{\Sigma}_{T}\bm{\Lambda}_{T}^{-1}\mathbf{Z}_{T}\left(\mathbf{Z}_{T}\bm{\Lambda}_{T}^{-1}\mathbf{Z}_{T}\right)^{-1}\text{,}
\]
which has the limit, as $T\rightarrow\infty$,
\begin{equation}
\left(\frac{1}{2\pi}\int_{-\pi}^{\pi}\frac{1}{f_{v}\left(\omega\right)}\text{d}\mathbf{H}\left(\omega\right)\right)^{-1}\left[\lim_{T\rightarrow\infty}\mathbf{Z}_{T}^{\top}\bm{\Lambda}_{T}^{-1}\bm{\Sigma}_{T}\bm{\Lambda}_{T}^{-1}\mathbf{Z}_{T}\right]\left(\frac{1}{2\pi}\int_{-\pi}^{\pi}\frac{1}{f_{v}\left(\omega\right)}\text{d}\mathbf{H}\left(\omega\right)\right)^{-1}\text{.}\label{eq: Limit1-1}
\end{equation}

Assumption {[}Main1{]} implies that $f_{u}$ is real and positive
since $\left\{ U_{t}\right\} _{t=-\infty}^{\infty}$ is an absolutely
summable linear filter of the independent finite variance sequence
$\left\{ E_{t}\right\} _{t=-\infty}^{\infty}$ (cf. Theorems 2.11
and 2.12 of \citealp{Fan2003}). Since $f_{v}$ is positive and continuous
by {[}Main3{]} and $f_{u}$ is real, positive and continuous by {[}Main1{]},
we can apply Lemma \ref{Lem: Ibragimov-1} to obtain
\begin{equation}
\lim_{T\rightarrow\infty}\mathbf{S}_{T}\text{var}\left[\tilde{\bm{\beta}}_{T}\left(f_{v}\right)\right]\mathbf{S}_{T}=\mathbf{C}_{v}\text{.}\label{eq: Limit 2}
\end{equation}
Upon substitution of (\ref{eq: Limit1-1}) into the left-hand side
(LHS) of (\ref{eq: Limit 2}) and rearrangement, we obtain

\begin{equation}
\lim_{T\rightarrow\infty}\mathbf{Z}_{T}^{\top}\bm{\Lambda}_{T}^{-1}\bm{\Sigma}_{T}\bm{\Lambda}_{T}^{-1}\mathbf{Z}_{T}=\frac{1}{2\pi}\int_{-\pi}^{\pi}\frac{f_{u}\left(\omega\right)}{f_{v}^{2}\left(\omega\right)}\text{d}\mathbf{H}\left(\omega\right)\text{,}\label{eq: Limit 3}
\end{equation}
and have thus verified (\ref{eq: Covariance Theorem}).
\end{proof}
\begin{thm}
\label{Thm: Main result}Under {[}Gren1{]}--{[}Gren4{]} and {[}Main1{]}--{[}Main4{]},
\begin{equation}
\mathbf{S}_{T}\left(\tilde{\bm{\beta}}_{T}\left(f_{v}\right)-\bm{\beta}\right)\overset{\mathcal{L}}{\longrightarrow}\text{N}\left(\mathbf{0},\mathbf{C}_{v}\right)\text{,}\label{eq: Asymptotic Normal}
\end{equation}
where $\mathbf{C}_{v}$ has the form (\ref{eq: Covariance Theorem}).
\end{thm}
\begin{proof}
It suffices to show that $\mathbf{w}_{T}$ is asymptotically normal
with mean $\mathbf{0}$ and covariance matrix equal to the LHS of
(\ref{eq: Limit 3}). First, write $\mathbf{w}_{T,N}=\mathbf{Z}_{T}^{-1}\bm{\Lambda}_{T}^{-1}\bm{\upsilon}_{T,N}$,
where
\[
\bm{\upsilon}_{T,N}^{\top}=\left(\sum_{i=-N}^{N}\theta_{i}E_{1-i},\sum_{i=-N}^{N}\theta_{i}E_{2-i},\dots,\sum_{i=-N}^{N}\theta_{i}E_{T-i-1},\sum_{i=-N}^{N}\theta_{i}E_{T-i}\right)\text{,}
\]
and $N=N\left(T\right)$ is a positive and increasing integer function
of $T$, such that $\lim_{T\rightarrow\infty}N\left(T\right)=\infty$.
Let $\bm{\upsilon}_{T,N}=\bm{\Upsilon}_{T,N}\mathbf{e}_{T,N}$, where

\[
\bm{\Upsilon}_{T,N}=\left[\begin{array}{cccccccccccccc}
0 & 0 & 0 & \cdots & 0 & \theta_{-N} & \theta_{-N+1} & \theta_{-N+2} & \cdots & \theta_{0} & \cdots & \theta_{N-2} & \theta_{N-1} & \theta_{N}\\
0 & 0 & \cdots & 0 & \theta_{-N} & \theta_{-N+1} & \theta_{-N+2} & \cdots & \theta_{0} & \cdots & \theta_{N-2} & \theta_{N-1} & \theta_{N} & 0\\
\vdots &  & \iddots & \iddots &  &  &  & \iddots &  &  &  & \iddots & \iddots & \vdots\\
0 & 0 & \theta_{-N} & \cdots & \theta_{-2} & \theta_{-1} & \theta_{0} & \theta_{1} & \theta_{2} & \cdots & \theta_{N} & 0 & \cdots & 0\\
0 & \theta_{-N} & \cdots & \theta_{-2} & \theta_{-1} & \theta_{0} & \theta_{1} & \theta_{2} & \cdots & \theta_{N} & 0 & \cdots & 0 & 0\\
\theta_{-N} & \cdots & \theta_{-2} & \theta_{-1} & \theta_{0} & \theta_{1} & \theta_{2} & \cdots & \theta_{N} & 0 & \cdots & 0 & 0 & 0
\end{array}\right]
\]
is a $T\times\left(T+2N-1\right)$ matrix and 
\[
\mathbf{e}_{T,N}^{\top}=\left(E_{1-N},E_{2-N},\dots,E_{T+N-1},E_{T+N}\right)
\]
is a $\left(T+2N-1\right)\times1$ vector.

To apply Lemma \ref{lem: billingsley-1}, we must show that for each
$N$, $\mathbf{w}_{T,N}$ converges in law to some $\bm{w}_{N}$,
as $T\rightarrow\infty$, where $\bm{\omega}_{N}$ is asymptotically
normal with mean $\mathbf{0}$ and covariance matrix (\ref{eq: Limit 3}),
as $N\rightarrow\infty$. Then, we must verify that 
\begin{equation}
\lim_{N\rightarrow\infty}\underset{T\rightarrow\infty}{\lim\sup}\text{ }\mathbb{P}\left(\left\Vert \mathbf{w}_{T,N}-\mathbf{w}_{T}\right\Vert _{2}\ge\varepsilon\right)=0\text{,}\label{eq: Probability bound}
\end{equation}
for each $\varepsilon>0$.

For the purpose of applying the Cramer-Wold device, define $\bm{\alpha}^{\top}=\left(\alpha_{1},\dots,\alpha_{d}\right)$,
where $\left\Vert \bm{\alpha}\right\Vert _{2}=1$. Let $\bm{\nu}_{k}$
denote the $k\text{th}$ column of $\text{\ensuremath{\mathbf{X}_{T}^{\top}\bm{\Lambda}_{T}^{-1}\bm{\Upsilon}_{T,N}}}$,
for $k\in\left\{ 1-N,\dots,T+N\right\} $. That is, 
\[
\mathbf{X}_{T}^{\top}\bm{\Lambda}_{T}^{-1}\bm{\Upsilon}_{T,N}=\left[\begin{array}{ccc}
\bm{\nu}_{1-N} & \cdots & \bm{\nu}_{T+N}\end{array}\right]\text{.}
\]
Therefore,

\begin{equation}
\bm{\alpha}^{\top}\mathbf{w}_{T,N}=\bm{\alpha}^{\top}\mathbf{S}_{T}^{-1}\sum_{t=1-N}^{T+N}\bm{\nu}_{t}E_{t}=\sigma\left[\sum_{t=1-N}^{T+N}\left(\bm{\alpha}^{\top}\mathbf{S}_{T}^{-1}\bm{\nu}_{t}\right)^{2}\right]^{1/2}\sum_{t=1-N}^{T+N}W_{t}\text{,}\label{eq: W to E}
\end{equation}
where 

\[
W_{t}=\frac{\bm{\alpha}^{\top}\mathbf{S}_{T}^{-1}\bm{\nu}_{t}}{\sigma\left[\sum_{k=1-N}^{T+N}\left(\bm{\alpha}^{\top}\mathbf{S}_{T}^{-1}\bm{\nu}_{k}\right)^{2}\right]^{1/2}}E_{t}\text{.}
\]

By {[}Main1{]} $\left\Vert \bm{\Upsilon}_{T,N}\right\Vert _{\text{op}}$
is bounded, and by {[}Gren4{]}, $\left\Vert \mathbf{X}_{T}\mathbf{S}^{-1}\right\Vert _{\text{op}}$
is bounded. Further, by {[}Main3{]} and {[}Main4{]}, we have the boundedness
of $\left\Vert \bm{\Lambda}_{T}^{-1}\right\Vert _{\text{op}}$. Thus,
we obtain the inequalities

\begin{equation}
0<\sum_{t=1-N}^{T+N}\left(\bm{\alpha}^{\top}\mathbf{S}_{T}^{-1}\bm{\nu}_{t}\right)^{2}\le\left(\left\Vert \mathbf{X}_{T}\mathbf{S}_{T}^{-1}\right\Vert _{\text{op}}\left\Vert \bm{\Lambda}_{T}^{-1}\right\Vert _{\text{op}}\left\Vert \bm{\Upsilon}_{T,N}\right\Vert _{\text{op}}\right)^{2}<\infty\text{.}\label{eq: bounded1}
\end{equation}
The last fact follows from an application of Lemma \ref{lem: Gray nonsing}
and all of the bounds are independent of $T$ and $N$.

Observe that $\left\{ W_{t}\right\} _{t=1-N}^{T+N}$ is a sequence
of independent random variables with expectation $\text{\ensuremath{\mathbb{E}}}\left(W_{t}\right)=0$
and $\sum_{t=1-N}^{T+N}\text{var}\left(W_{t}\right)=1$. Let $F_{W_{t}}\left(w\right)$
be the distribution function of $W_{t}$, for each $t$. Then, for
any $\delta>0$, we have the bound:

\begin{align}
\sum_{t=1-N}^{T+N}\int_{\left|w\right|>\delta}w^{2}\text{d}F_{W_{t}}\left(w\right) & =\frac{1}{\sigma^{2}}\sum_{t=1-N}^{T+N}\frac{\left(\bm{\alpha}^{\top}\mathbf{S}_{T}^{-1}\bm{\nu}_{t}\right)^{2}}{\sum_{k=1-N}^{T+N}\left(\bm{\alpha}^{\top}\mathbf{S}_{T}^{-1}\bm{\nu}_{k}\right)^{2}}\int_{\left|e\right|>\delta e_{T,N}^{t}}e^{2}\text{d}F_{E_{t}}\left(e\right)\nonumber \\
 & \le\frac{1}{\sigma^{2}}\underset{t\in\left\{ 1-N,\dots,T+N\right\} }{\sup}\int_{\left|e\right|>\delta e_{T,N}^{*}}e^{2}\text{d}F_{E_{t}}\left(e\right)\text{,}\label{eq: bounded2}
\end{align}
where
\[
e_{T,N}^{t}=\frac{\sigma\left[\sum_{k=1-N}^{T+N}\left(\bm{\alpha}^{\top}\mathbf{S}_{T}^{-1}\bm{\nu}_{k}\right)^{2}\right]^{1/2}}{\left|\bm{\alpha}^{\top}\mathbf{S}_{T}^{-1}\bm{\nu}_{t}\right|}\text{, and }e_{T,N}^{*}=\frac{\sigma\left[\sum_{t=1-N}^{T+N}\left(\bm{\alpha}^{\top}\mathbf{S}_{T}^{-1}\bm{\nu}_{t}\right)^{2}\right]^{1/2}}{\underset{t\in\left\{ 1-N,\dots,T+N\right\} }{\sup}\left|\bm{\alpha}^{\top}\mathbf{S}_{T}^{-1}\bm{\nu}_{t}\right|}\text{.}
\]

By the bound in (\ref{eq: bounded1}) and {[}Main2{]}, we must show
that
\begin{equation}
\underset{t\in\left\{ 1-N,\dots,T+N\right\} }{\sup}\left|\bm{\alpha}^{\top}\mathbf{S}_{T}^{-1}\bm{\nu}_{t}\right|\rightarrow0\text{,}\label{eq: sup abs}
\end{equation}
to prove that (\ref{eq: bounded2}) converges to zero, as $T\rightarrow\infty$.
Write the $i\text{th}$ row and $j\text{th}$ column element of $\bm{\Upsilon}_{T,N}$
and $\bm{\Lambda}_{T}^{-1}$ as $\Upsilon_{ij}$ and $\lambda_{ij}=\lambda_{\left|i-j\right|}$,
respectively. Upon expansion we can obtain the following inequalities
for the LHS of (\ref{eq: sup abs}):

\begin{align*}
\underset{t\in\left\{ 1-N,\dots,T+N\right\} }{\sup}\left|\bm{\alpha}^{\top}\mathbf{S}_{T}^{-1}\bm{\nu}_{t}\right| & \le\frac{1}{\sigma^{2}}\sup_{t}\left|\sum_{i=1}^{d}\frac{\alpha_{i}}{s_{iT}}\sum_{j=1}^{T}\Upsilon_{jt}\sum_{k=1}^{T}\lambda_{\left|k-j\right|}x_{ik}\right|\\
 & \le C\max_{i\in\left[d\right]}\frac{\max_{k\in\left[T\right]}\left|x_{ik}\right|}{s_{iT}}\sup_{t}\sum_{j=1}^{T}\left|\Upsilon_{jt}\right|\sum_{k=0}^{\infty}\left|\lambda_{k}\right|\text{,}
\end{align*}
where $C>0$ is some finite constant.

By {[}Main1{]}, $\left\Vert \bm{\Upsilon}_{T,N}\right\Vert _{1}$
and $\left\Vert \bm{\Upsilon}_{T,N}\right\Vert _{\infty}$ are bounded,
independently of $T$ and $N$.Therefore, for any $T$ and $N$, $\sup_{t}\sum_{j=1}^{T}\left|\Upsilon_{jt}\right|<\infty$.
Similarly, by Assumptions {[}Main3{]} and {[}Main4{]}, we apply Corollary
\ref{cor: bound the inverse} to show that $\left\Vert \bm{\Lambda}_{T}^{-1}\right\Vert _{1},\left\Vert \bm{\Lambda}_{T}^{-1}\right\Vert _{\infty}<\infty$,
independently of $T$, and thus $\sum_{k=0}^{\infty}\left|\lambda_{k}\right|<\infty$.
Lastly,

\[
\lim_{T\rightarrow\infty}\max_{i\in\left[d\right]}\frac{\max_{k\in\left[T\right]}\left|x_{ik}\right|}{s_{iT}}=0\text{,}
\]
by {[}Gren1{]}, {[}Gren2{]}, and \citet[Lem. 2.6.1]{Anderson1971}.
Thus (\ref{eq: sup abs}) is proved.

Next, (\ref{eq: sup abs}) is sufficient to guarantee that (\ref{eq: bounded2})
approaches zero as $T$ approaches infinity. We can apply the Lindeberg-Feller
central limit theorem \citep[Thm. 5.1]{DasGupta2008} to obtain $\sum_{t=1-N}^{T+N}W_{t}\overset{\mathcal{L}}{\longrightarrow}\text{N}\left(0,1\right)$.

Via (\ref{eq: W to E}), $\bm{\alpha}^{\top}\bm{w}_{T,N}$ is asymptotically
equal in distribution to $\bm{\alpha}^{\top}\bm{\omega}_{N}$ (as
$T\rightarrow\infty$), where, for any choice of $\bm{\alpha}$, $\bm{\alpha}^{\top}\bm{w}_{N}$
is normal with mean zero and variance
\[
\lim_{T\rightarrow\infty}\bm{\alpha}^{\top}\mathbb{E}\left(\mathbf{w}_{T,N}\mathbf{w}_{T,N}^{\top}\right)\bm{\alpha}\text{.}
\]

Via the Cramer-Wold device (cf. \citealp[Th. 1.16]{DasGupta2008}),
$\bm{w}_{T,N}$ is asymptotically normal with mean vector $\mathbf{0}$
and covariance
\begin{align}
\mathbb{E}\left(\bm{w}_{N}\bm{w}_{N}^{\top}\right) & =\lim_{T\rightarrow\infty}\mathbb{E}\left(\mathbf{w}_{T,N}\mathbf{w}_{T,N}^{\top}\right)\nonumber \\
 & =\lim_{T\rightarrow\infty}\mathbf{Z}_{T}^{\top}\bm{\Lambda}_{T}^{-1}\bm{\Sigma}_{T,N}\bm{\Lambda}_{T}^{-1}\mathbf{Z}_{T}\nonumber \\
 & =\frac{1}{2\pi}\int_{-\pi}^{\pi}\frac{f_{u,N}\left(\omega\right)}{f_{v}^{2}\left(\omega\right)}\text{d}\mathbf{H}\left(\omega\right)\text{,}\label{eq: W_TN to W_N}
\end{align}
using Lemma \ref{Lem: Ibragimov-1}, where $f_{u,N}$ is the SDF corresponding
to $\bm{\Sigma}_{T,N}=\mathbb{E}\left(\bm{\upsilon}_{T,N}\bm{\upsilon}_{T,N}^{\top}\right)$.
In other words, $\bm{w}_{N}$ is normally distributed with zero mean
vector and covariance matrix (\ref{eq: W_TN to W_N}).

By {[}Main1{]}, $f_{u}\left(\omega\right)=\left(\sigma^{2}/2\pi\right)\left|\sum_{i=-\infty}^{\infty}\theta_{i}e^{-\iota i\omega}\right|^{2}$
and $f_{u,N}\left(\omega\right)=\left(\sigma^{2}/2\pi\right)\left|\sum_{i=-N}^{N}\theta_{i}e^{-\iota i\omega}\right|^{2}$,
via the power transfer formula (cf. Lemma \ref{lem fu rudin}). Furthermore,
since $\sum_{i=-\infty}^{\infty}\left|\theta_{i}\right|<\infty$ and
by preservation of uniform convergence under continuous composition
\citep{Bartle1961}, $f_{u,N}$ converges uniformly to $f_{u}$, as
$N$ approaches infinity (cf. \citealp[Sec. 4.1]{Gray2006}). Via
the Cramer-Wold device, $\bm{w}_{N}$ converges in law to $\bm{w}$
($N\rightarrow\infty$), where $\bm{\alpha}^{\top}\bm{w}$ has mean
vector $\mathbf{0}$ and covariance matrix
\begin{align}
\mathbb{E}\left(\bm{w}\bm{w}^{\top}\right) & =\lim_{T\rightarrow\infty}\mathbf{Z}_{T}^{\top}\bm{\Lambda}_{T}^{-1}\bm{\Sigma}_{N}\bm{\Lambda}_{T}^{-1}\mathbf{Z}_{T}\nonumber \\
 & =\frac{1}{2\pi}\int_{-\pi}^{\pi}\frac{f_{u}\left(\omega\right)}{f_{v}^{2}\left(\omega\right)}\text{d}\mathbf{H}\left(\omega\right)\text{,}\label{eq: W_N to W}
\end{align}
which is equal to (\ref{eq: Limit 3}).

Finally, we must verify (\ref{eq: Probability bound}). We use Chebyshev's
inequality, which states that for any $\varepsilon>0$,
\begin{align}
\mathbb{P}\left(\left\Vert \mathbf{w}_{T,N}-\mathbf{w}_{T}\right\Vert _{2}\ge\varepsilon\right) & \le\frac{\mathbb{E}\left(\left\Vert \mathbf{w}_{T,N}-\mathbf{w}_{T}\right\Vert _{2}^{2}\right)}{\varepsilon^{2}}\text{,}\label{eq: cheby}
\end{align}
where we write the numerator of the right-hand side of (\ref{eq: cheby})
as
\begin{align*}
 & \text{tr}\left[\left(\mathbf{Z}_{T}^{-1}\bm{\Lambda}_{T}^{-1}\bm{\upsilon}_{T,N}-\mathbf{Z}_{T}^{-1}\bm{\Lambda}_{T}^{-1}\mathbf{u}_{T}\right)\left(\mathbf{Z}_{T}^{-1}\bm{\Lambda}_{T}^{-1}\bm{\upsilon}_{T,N}-\mathbf{Z}_{T}^{-1}\bm{\Lambda}_{T}^{-1}\mathbf{u}_{T}\right)^{\top}\right]\\
= & \text{tr}\left\{ \mathbf{Z}_{T}^{-1}\bm{\Lambda}_{T}^{-1}\mathbb{E}\left[\left(\bm{\upsilon}_{T,N}-\mathbf{u}_{T}\right)\left(\bm{\upsilon}_{T,N}-\mathbf{u}_{T}\right)^{\top}\right]\bm{\Lambda}_{T}^{-1}\mathbf{Z}_{T}^{-1}\right\} \text{,}
\end{align*}
which reduces to
\[
\text{tr}\left\{ \mathbf{Z}_{T}^{-1}\bm{\Lambda}_{T}^{-1}\mathbb{E}\left[\mathbf{u}_{T}\mathbf{u}_{T}^{\top}\right]\bm{\Lambda}_{T}^{-1}\mathbf{Z}_{T}^{-1}\right\} -\text{tr}\left\{ \mathbf{Z}_{T}^{-1}\bm{\Lambda}_{T}^{-1}\mathbb{E}\left[\bm{\upsilon}_{T,N}\bm{\upsilon}_{T,N}^{\top}\right]\bm{\Lambda}_{T}^{-1}\mathbf{Z}_{T}^{-1}\right\} \text{.}
\]

By (\ref{eq: Limit 3}), (\ref{eq: W_TN to W_N}), and (\ref{eq: W_N to W}),
we have
\[
\lim_{N\rightarrow\infty}\lim_{T\rightarrow\infty}\text{tr}\left\{ \mathbf{Z}_{T}^{-1}\bm{\Lambda}_{T}^{-1}\mathbb{E}\left[\mathbf{u}_{T}\mathbf{u}_{T}^{\top}\right]\bm{\Lambda}_{T}^{-1}\mathbf{Z}_{T}^{-1}\right\} -\text{tr}\left\{ \mathbf{Z}_{T}^{-1}\bm{\Lambda}_{T}^{-1}\mathbb{E}\left[\bm{\upsilon}_{T,N}\bm{\upsilon}_{T,N}^{\top}\right]\bm{\Lambda}_{T}^{-1}\mathbf{Z}_{T}^{-1}\right\} =0\text{.}
\]
Thus, by (\ref{eq: cheby}), condition (\ref{eq: Probability bound})
is verified. This completes the proof.
\end{proof}

\section{Discussions and remarks}

\subsection{Notes regarding the assumptions of Theorem \ref{Thm: Main result}}

We can directly compare {[}Main1{]} to {[}Hann1{]}. It is notable
that {[}Hann1{]} is more general than {[}Main1{]} since it allows
$\left\{ U_{t}\right\} _{t=-\infty}^{\infty}$ to be a linear filter
over a martingale sequence $\left\{ E_{t}\right\} _{t=-\infty}^{\infty}$,
satisfying $\mathbb{E}\left(E_{t}|\mathcal{F}_{t-1}\right)=\mathbb{E}\left(E_{t}^{2}-\mathbb{E}\left(E_{t}^{2}\right)|\mathcal{F}_{t-1}\right)=0$.
Further, the condition $0<\left|\sum_{i=-\infty}^{\infty}\theta_{i}e^{-\iota i\omega}\right|$
is necessitated so that Lemma \ref{Lem: Ibragimov-1} can be applied.
It is remarked in \citet{Amemiya1973}, however, that {[}Main1{]}
and {[}Main2{]} are more general than {[}Amem1{]}.

The addition {[}Main4{]} is the key that facilitates the proof. This
assumption is necessary for bounding $\left\Vert \bm{\Lambda}_{T}^{-1}\right\Vert _{1}$
and $\left\Vert \bm{\Lambda}_{T}^{-1}\right\Vert _{\infty}$, which
is required to prove (\ref{eq: Asymptotic Normal}). It must be remarked
that {[}Main4{]} is a common condition in the literature, and has
been made in similar proof methods, such as those of \citet{Cheng2015}.
The assumption is not restrictive, since a broad class of short-memory
processes satisfy {[}Main4{]}. For example, any stationary autoregressive
moving average (ARMA) process will satisfy {[}Main4{]} (cf. \citealp[Sec. 2.5]{Fan2003}).
The assumption is also commonly used in the analysis of unit root
processes (see, e.g., \citealp[17.5]{Hamilton1994}).

\subsection{Feasible generalized least squares}

Generally the SDF $f_{v}$ is unknown and must be estimated from data.
Suppose that $\hat{f}_{v,T}$ is an estimator of $f_{v}$, which is
indexed by the sample size $T$. Denote the FGLS estimator of $\bm{\beta}$
by $\tilde{\bm{\beta}}_{T}\left(\hat{f}_{v,T}\right)$. For the FGLS
to be of use, we require that the FGLS has the same asymptotic distribution
as $\tilde{\bm{\beta}}_{T}\left(f_{v}\right)$. To this end, it is
sufficient to show that
\begin{equation}
\mathbf{S}_{T}\tilde{\bm{\beta}}_{T}\left(\hat{f}_{v,T}\right)\overset{\mathbb{P}}{\longrightarrow}\mathbf{S}_{T}\tilde{\bm{\beta}}_{T}\left(f_{v}\right)\text{,}\label{eq: convergence in prob}
\end{equation}
where $\overset{\mathbb{P}}{\longrightarrow}$ denotes convergence
in probability. Denote the $T\times T$ auto-covariance matrix corresponding
to the estimator $\hat{f}_{v,T}$ as $\hat{\bm{\Lambda}}_{T}$. Then,
we may write $\tilde{\bm{\beta}}_{T}\left(\hat{\bm{\Lambda}}_{T}\right)=\tilde{\bm{\beta}}_{T}\left(\hat{f}_{v,T}\right)$.

Under {[}Gren1{]}--{[}Gren4{]} and {[}Amem1{]}, \citet[Thm. 2]{Amemiya1973}
proved that if $f_{v}$ is the SDP of an AR process of order $N\in\mathbb{N}$,
that satisfies {[}Amem2{]}, then (\ref{eq: convergence in prob})
holds, when $\hat{\bm{\Lambda}}_{T}$ is obtained via the OLS estimator
for the AR model coefficients (cf. \citealp[Sec. 5.4]{Amemiya1985}).
The argument from \citet[Thm. 2]{Amemiya1973} would hold whenever
$\hat{\bm{\Lambda}}_{T}$ is obtained via any consistent estimator
of the AR model coefficients.

The proof of the theorem also remains the same upon replacing {[}Amem1{]}
by {[}Main1{]} and {[}Main2{]} and noting that {[}Amem2{]} is implied
by {[}Main3{]} and {[}Main4{]}. Thus, under the hypothesis of Theorem
\ref{Thm: Main result}, if $f_{v}$ is hypothesized to be the SDF
of a stationary AR process of order $N\in\mathbb{N}$ (i.e., satisfying
{[}Amem2{]}), then (\ref{eq: convergence in prob}) holds, where $\hat{\bm{\Lambda}}_{T}$
is obtained via any consistent estimator of the AR coefficients.

It is notable that proving that (\ref{eq: convergence in prob}) holds,
under {[}Amem2{]} is permissive due to the fact that the inverse auto-covariance
matrix $\bm{\Lambda}_{T}^{-1}$ has a banded Toeplitz form (cf. \citealp{Verbyla:1985aa}).
We conjecture that it is possible to obtain similar results using
the same techniques as those from \citet{Amemiya1973}, when $f_{v}$
is any parametric family of SDFs with banded Toeplitz inverse auto-covariance
matrices $\bm{\Lambda}_{T}^{-1}$. However, the proof of such a result
is beyond the scope of the current paper.

\subsection{Further comments regarding the frequency-domain IGLS estimator}

We note that \citet{Hannan1973} proved a more general result than
that which we reported in Section \ref{sec:Introduction}. Consider
the following conditions.
\begin{lyxlist}{00.00.0000}
\item [{{[}Hann4{]}}] The sequence $\left\{ U_{t}\right\} _{t=-\infty}^{\infty}$
is stationary and $\alpha\text{-mixing}$ with $\sum_{t=1}^{\infty}\alpha_{t}^{2/\left(2+\delta\right)}<\infty$,
where $\alpha_{t}$ is the $t\text{th}$ $\alpha\text{-mixing}$ number
and $\mathbb{E}\left|U_{t}\right|^{2+\delta}<\infty$ for some $\delta>0$.
\item [{{[}Hann5{]}}] The sequence $\left\{ \mathbf{x}_{t}\right\} _{t=1}^{\infty}$
is a finite segment of $\left\{ \mathbf{x}_{t}\right\} _{t=-\infty}^{\infty}$,
where $\left\{ \mathbf{x}_{t}\right\} _{t=-\infty}^{\infty}$, is
strictly stationary, ergodic, and independent of $\left\{ U_{t}\right\} _{t=-\infty}^{\infty}$.
\end{lyxlist}
From \citet{Hannan1973}, it was proved that (\ref{eq: AN FD}) could
be obtained by assuming either {[}Gren1{]}--{[}Gren4{]} or {[}Hann5{]},
and either {[}Hann4{]}, or {[}Hann1{]} and {[}Hann2{]}, together with
{[}Hann3{]}. In the case where {[}Hann5{]} is assumed, the matrix
$\mathbf{R}\left(k\right)$ in {[}Gren4{]} is replaced by $\bar{\mathbf{R}}\left(k\right)$,
with elements $\bar{\rho}_{ij}\left(k\right)=\mathbb{E}\left(X_{it}X_{j,t+k}\right)/\left[\mathbb{E}\left(X_{it}^{2}\right)\mathbb{E}\left(X_{j,t}^{2}\right)\right]^{1/2}$,
where $\mathbf{x}_{t}^{\top}=\left(X_{1t},\dots X_{pt}\right)$ is
now stochastic.

Assumptions {[}Hann4{]} and {[}Hann5{]} are similar to the mixing
and stochastic assumptions considered in the proofs of \citet{Kapetanios2016}.
A proof of the main result under these conditions could be thus adapted
from the techniques of \citet{Kapetanios2016}. We believe that this
is an interesting direction of research. However, it falls outside
the aim of our paper, which was to directly improve upon the IGLS
results obtained in \citet{Amemiya1973}.

\section*{Appendix}

The following results are required in our main proofs. Sources for
all unproved results are provided at the end of the section.
\begin{lem}
\label{Lem: Anderson 1} Let $\mathbf{v}_{T}^{\top}=\left(V_{1},\dots,V_{T}\right)$
have covariance matrix $\mathbb{E}\left(\mathbf{v}_{T}\mathbf{v}_{T}^{\top}\right)=\bm{\Lambda}_{T}$,
where $\left\{ V_{t}\right\} _{t=1}^{T}$ is a finite segment of the
random sequence $\left\{ V_{t}\right\} _{t=-\infty}^{\infty}$ with
SDF $f_{v}$. Under {[}Gren1{]}--{[}Gren4{]}, and {[}Main3{]},
\[
\lim_{T\rightarrow\infty}\mathbf{Z}_{T}^{\top}\bm{\Lambda}_{T}^{-1}\mathbf{Z}_{T}=\frac{1}{2\pi}\int_{-\pi}^{\pi}\frac{1}{f_{v}\left(\omega\right)}\text{d}\mathbf{H}\left(\omega\right)\text{.}
\]
\end{lem}
\begin{lem}
\label{lem: Gray nonsing}Let $\mathbf{v}_{T}^{\top}=\left(V_{1},\dots,V_{T}\right)$
have covariance matrix $\mathbb{E}\left(\mathbf{v}_{T}\mathbf{v}_{T}^{\top}\right)=\bm{\Lambda}_{T}$,
where $\left\{ V_{t}\right\} _{t=1}^{T}$ is a finite segment of the
random sequence $\left\{ V_{t}\right\} _{t=-\infty}^{\infty}$ with
SDF $f_{v}$. If $f_{v}\left(\omega\right)\ge c>0$ for all $\omega\in\left[-\pi,\pi\right]$,
then $\bm{\Lambda}_{T}$ is nonsingular and $\left\Vert \bm{\Lambda}_{T}^{-1}\right\Vert _{\text{op}}<\infty$.
\end{lem}
\begin{lem}
\label{lem fu rudin}Let the random sequence $\left\{ U_{t}\right\} _{t=-\infty}^{\infty}$
satisfy {[}Main1{]} and have SDF $f_{u}$. Then $f_{u}\left(\omega\right)$
is real, positive, and continuous.
\end{lem}
\begin{proof}
Form (\ref{eq: power transfer}) implies that $\left\{ U_{t}\right\} _{t=-\infty}^{\infty}$
is a linear filter of $\left\{ E_{t}\right\} _{t=-\infty}^{\infty}$.
We may write the so-called transfer function of this filtering as
$\Gamma\left(\omega\right)=\sum_{i=-\infty}^{\infty}\theta_{i}e^{-\iota i\omega}$.
Since $\sum_{i=-\infty}^{\infty}\left|\theta_{i}\right|<\infty$,
we have the fact that $\Gamma\left(\omega\right)$ is continuous (cf.
\citealp[Sec. 9.4]{Rudin:1987aa}). By the same fact, $\Gamma\left(\omega\right)$
is also bounded in modulus (i.e. $\left|\Gamma\right|<\infty$). Next,
using Theorem 2.12 of \citet{Fan2003}, we may write
\[
f_{u}\left(\omega\right)=\frac{\sigma^{2}}{2\pi}\left|\Gamma\left(\omega\right)\right|^{2}\text{,}
\]
for each $\omega\in\left[-\pi,\pi\right]$, by the fact that $\left\{ E_{t}\right\} _{t=-\infty}^{\infty}$
is an independent sequence with $\mathbb{E}\left(E_{t}\right)=0$
and $\text{var}\left(E_{t}\right)=\sigma^{2}$. Since the squared
modulus is real and positive, by our assumptions, we have the desired
result.
\end{proof}
\begin{lem}
\label{Lem: Ibragimov-1} Under {[}Gren1{]}--{[}Gren4{]}, if the
sequences $\left\{ U_{t}\right\} _{t=-\infty}^{\infty}$ and $\left\{ V_{t}\right\} _{t=-\infty}^{\infty}$
have positive and continuous SDFs $f_{u}$ and $f_{v}$, respectively,
and if $f_{u}$ is bounded, then
\[
\lim_{T\rightarrow\infty}\mathbf{S}_{T}\text{var}\left[\tilde{\bm{\beta}}_{T}\left(f_{v}\right)\right]\mathbf{S}_{T}=\mathbf{C}_{v}\text{.}
\]
\end{lem}
\begin{lem}
\label{lem: billingsley-1}For each $N$, suppose that $\bm{X}_{T,N}\overset{\mathcal{L}}{\longrightarrow}\bm{X}_{N}$,
as $T\rightarrow\infty$, and that $\bm{X}_{N}\overset{\mathcal{L}}{\longrightarrow}\bm{X}$,
as $N\rightarrow\infty$. Furthermore, assume that
\[
\lim_{N\rightarrow\infty}\limsup_{T\rightarrow\infty}\mathbb{P}\left(\left\Vert \bm{X}_{T,N}-\bm{Y}_{T}\right\Vert \ge\epsilon\right)=0\text{,}
\]
for each $\epsilon>0$ and some $\left\{ \bm{Y}_{T}\right\} $. If
each of the random variables involved have a common separable domain,
and if $\left\Vert \cdot\right\Vert $ is some appropriate norm on
said domain, then $\bm{Y}_{T}\overset{\mathcal{L}}{\longrightarrow}\bm{X}$,
as $T\rightarrow\infty$.
\end{lem}
\begin{lem}
\label{lem: Brillinger}Let $\left\{ V_{t}\right\} _{t=-\infty}^{\infty}$
be a real-valued, where $\mathbb{E}\left(V_{t}\right)=0$ ($t\in\mathbb{Z}$),
$\eta_{i}=\text{cov}\left(V_{t},V_{t+i}\right)$ ($i\in\mathbb{Z}$;
$\eta_{-i}=\eta_{i}$), and SDF $f_{v}\left(\omega\right)=\left(2\pi\right)^{-1}\sum_{i=-\infty}^{\infty}\eta_{i}e^{-\iota i\omega}$.
If $f_{v}$ is positive and $\sum_{i=-\infty}^{\infty}i\left|\eta_{i}\right|<\infty$,
then, $V_{t}=\sum_{i=0}^{\infty}b_{i}E_{t-i}$, where $E_{t}=\sum_{i=0}^{\infty}a_{i}V_{t-i}$
is a random process, with $\mathbb{E}\left(E_{t}\right)=0$ and $\text{var}\left(E_{t}\right)=\varsigma^{2}>0$
($a_{0}=1$ and $b_{0}=1$). Furthermore, $\sum_{i=0}^{\infty}i\left|a_{i}\right|<\infty$,
$\sum_{i=0}^{\infty}i\left|b_{i}\right|<\infty$, and
\[
f_{v}\left(\omega\right)=\frac{\varsigma^{2}}{2\pi}\left|b\left(\omega\right)\right|^{2}=\frac{\varsigma^{2}}{2\pi}\left|a\left(\omega\right)\right|^{-2}\text{,}
\]
where $a\left(\omega\right)=\sum_{i=0}^{\infty}a_{i}e^{-\iota i\omega}$
and $b\left(\omega\right)=\sum_{i=0}^{\infty}b_{i}e^{-\iota i\omega}$.
\end{lem}
Let $\bm{\Lambda}_{T}$ be a covariance matrix that satisfies Lemma
\ref{lem: Brillinger}. Using the Cholesky decomposition of \citet{Akaike1969},
we may write the inverse of $\bm{\Lambda}_{T}$ in the form $\bm{\Lambda}_{T}^{-1}=\mathbf{L}_{T}^{\top}\bm{\Delta}_{T}^{-1}\mathbf{L}_{T}$,
where
\[
\mathbf{L}_{T}=\left[\begin{array}{ccccc}
1 & 0 & 0 & \cdots & 0\\
a_{11} & 1 & 0 & \cdots & 0\\
a_{22} & a_{12} & 1 &  & \vdots\\
\vdots & \vdots & \ddots & \ddots & 0\\
a_{T-1,T-1} & a_{T-2,T-1} & \cdots & a_{1,T-1} & 1
\end{array}\right]
\]
and $\bm{\Delta}_{T}=\text{diag}\left(\varsigma_{0}^{2},\varsigma_{1}^{2},\dots,\varsigma_{T-1}^{2}\right)$.
Here, $a_{ij}$ and $\varsigma_{k}^{2}$ ($i,j,k\in\left[T-1\right]$)
are obtained by solving the following problems: for each $k$
\[
\left(a_{1k},\dots a_{kk}\right)=\arg\underset{\left(c_{1},\dots,c_{k}\right)}{\min}\text{ }\mathbb{E}\left[\left(V_{t}+c_{1}V_{t-1}+\dots+c_{k}V_{t-k}\right)^{2}\right]\text{,}
\]
and $\varsigma_{k}^{2}>0$ is the minimal value of the problem. When
$k=0$, set $\varsigma_{0}^{2}=\varsigma^{2}$. We have the following
result.
\begin{lem}
\label{lem: Berk}Let $\left\{ V_{t}\right\} _{t=-\infty}^{\infty}$
be the time series from Lemma \ref{lem: Brillinger}. Then, there
exists a $J$ such that for any $j\ge J$, $\sum_{i=1}^{j}\left|a_{ij}-a_{i}\right|\le C\sum_{i=j+1}^{\infty}\left|a_{i}\right|$,
where $C>0$ is a finite constant that only depends on $f_{v}$, and
$a_{ij}$ is the $i\text{th}$ row and $j\text{th}$ column element
of the matrix $\mathbf{L}_{T}$.
\end{lem}
\begin{cor}
\label{cor: bound the inverse}Let $\left\{ V_{t}\right\} _{t=-\infty}^{\infty}$
be the time series from Lemma \ref{lem: Brillinger}. Then, $\left\Vert \bm{\Lambda}_{T}^{-1}\right\Vert _{1},\left\Vert \bm{\Lambda}_{T}^{-1}\right\Vert _{\infty}<\infty$,
independently of $T$.
\end{cor}
\begin{proof}
Write $\bm{\Lambda}_{T}^{-1}=\mathbf{L}_{T}^{\top}\bm{\Delta}_{T}^{-1}\mathbf{L}_{T}$
and apply \citet[Thm. 4.6.5]{Seber2008} to obtain

\[
\left\Vert \bm{\Lambda}_{T}^{-1}\right\Vert _{1}\le\left\Vert \mathbf{L}_{T}^{\top}\right\Vert _{1}\left\Vert \bm{\Delta}_{T}^{-1}\right\Vert _{1}\left\Vert \mathbf{L}_{T}\right\Vert _{1}\text{ and }\left\Vert \bm{\Lambda}_{T}^{-1}\right\Vert _{\infty}\le\left\Vert \mathbf{L}_{T}^{\top}\right\Vert _{\infty}\left\Vert \bm{\Delta}_{T}^{-1}\right\Vert _{\infty}\left\Vert \mathbf{L}_{T}\right\Vert _{\infty}\text{.}
\]
Note that $\left\Vert \mathbf{L}_{T}^{\top}\right\Vert _{1}=\left\Vert \mathbf{L}_{T}\right\Vert _{\infty}$,
$\left\Vert \mathbf{L}_{T}^{\top}\right\Vert _{\infty}=\left\Vert \mathbf{L}_{T}\right\Vert _{1}$,
and $\left\Vert \bm{\Delta}_{T}^{-1}\right\Vert _{1}=\left\Vert \bm{\Delta}_{T}^{-1}\right\Vert _{\infty}$.
Thus we require only three computations.

Let $a_{0k}=1$ and $a_{0}=1$ and write

\[
\left\Vert \mathbf{L}_{T}\right\Vert _{\infty}=\max_{j\in\left[T-1\right]}\left\{ \left|a_{0k}\right|+\left|a_{1j}\right|+\dots+\left|a_{jj}\right|\right\} \text{.}
\]
By the triangle inequality, we obtain
\[
\sum_{i=0}^{j}\left|a_{ij}\right|=\sum_{i=0}^{j}\left|a_{ij}-a_{i}+a_{i}\right|\le\sum_{i=0}^{j}\left|a_{ij}-a_{i}\right|+\sum_{i=0}^{j}\left|a_{i}\right|\text{,}
\]
which, for sufficiently large $T$, allows for the application of
Lemma \ref{lem: Berk} in order to obtain
\[
\sum_{i=0}^{j}\left|a_{ij}\right|\le\sum_{i=0}^{j}\left|a_{i}\right|+C\sum_{i=j+1}^{\infty}\left|a_{i}\right|\text{,}
\]
where $C>0$ is a finite constant that is independent of $T$. Since
$\sum_{i=0}^{\infty}i\left|a_{i}\right|<\infty$, we have the fact
that $\sum_{i=0}^{\infty}i\left|a_{i}\right|<\infty$ and $\sum_{i=0}^{j}\left|a_{ij}\right|<\infty$,
for any $j\in\left[T-1\right]$ . Thus, $\left\Vert \mathbf{L}_{T}\right\Vert _{\infty}$
is bounded, independently of $T$.

Next, we write $\left\Vert \mathbf{L}_{T}\right\Vert _{1}$ as
\begin{equation}
\left\Vert \mathbf{L}_{T}\right\Vert _{1}=\max\left(1+\left\{ \begin{array}{c}
0\\
\left|a_{1,T-1}\right|\\
\left|a_{2,T-1}\right|+\left|a_{1,T-2}\right|\\
\left|a_{3,T-1}\right|+\left|a_{2,T-2}\right|+\left|a_{1,T-3}\right|\\
\vdots\\
\left|a_{T-1,T-1}\right|+\left|a_{T-2,T-2}\right|+\dots+\left|a_{22}\right|+\left|a_{11}\right|
\end{array}\right\} \right)\text{.}\label{eq: L1}
\end{equation}

Using Lemma \ref{lem: Berk}, we can bound $\left|a_{1,T-1}\right|$
by $\sum_{i=1}^{T-1}\left|a_{i,T-1}-a_{i}\right|\le C\sum_{i=T}^{\infty}\left|a_{i}\right|$,
which convergences, since $\sum_{i=0}^{\infty}\left|a_{i}\right|<\infty$.
Similarly, we can use Lemma \ref{lem: Berk} to bound $\left|a_{2,T-1}\right|+\left|a_{1,T-2}\right|$
by

\begin{align*}
\sum_{i=1}^{T-1}\left|a_{i,T-1}-a_{i}\right|+\sum_{i=1}^{T-2}\left|a_{i,T-2}-a_{i}\right| & \le C\sum_{i=T}^{\infty}\left|a_{i}\right|+C\sum_{i=T-1}^{\infty}\left|a_{i}\right|\\
 & =C\left(\sum_{i=T}^{\infty}2\left|a_{i}\right|+\left|a_{T-1}\right|\right)\text{.}
\end{align*}
This converges, since $\sum_{i=0}^{\infty}\left|a_{i}\right|<\infty$.
Continuing the pattern, we arrive at the final expression
\[
\left|a_{T-1,T-1}\right|+\left|a_{T-2,T-2}\right|+\dots+\left|a_{22}\right|+\left|a_{11}\right|\text{,}
\]
which can be bounded by
\begin{equation}
\sum_{i=1}^{T-1}\left|a_{i,T-1}-a_{i}\right|+\sum_{i=1}^{T-2}\left|a_{i,T-2}-a_{i}\right|+\dots+\sum_{i=1}^{T-2}\left|a_{i,T-2}-a_{i}\right|\text{.}\label{eq: pre berk}
\end{equation}
Applying Lemma \ref{lem: Berk}, we bound (\ref{eq: pre berk}) by
\begin{eqnarray*}
 &  & C\sum_{i=T}^{\infty}\left|a_{i}\right|+C\sum_{i=T-1}^{\infty}\left|a_{i}\right|+\dots+C\sum_{i=2}^{\infty}\left|a_{i}\right|+C\sum_{i=1}^{\infty}\left|a_{i}\right|\\
 & = & C\left(\sum_{i=T}^{\infty}\left|a_{i}\right|+\sum_{i=T-1}^{\infty}\left|a_{i}\right|+\dots+\sum_{i=2}^{\infty}\left|a_{i}\right|+\sum_{i=1}^{\infty}\left|a_{i}\right|\right)\\
 & = & C\left(\sum_{i=T}^{\infty}T\left|a_{i}\right|+\left(T-1\right)\left|a_{T-1}\right|+\dots+2\left|a_{2}\right|+\left|a_{1}\right|\right)\\
 & \le & C\sum_{i=1}^{\infty}i\left|a_{i}\right|\text{.}
\end{eqnarray*}
The expression $C\sum_{i=1}^{\infty}i\left|a_{i}\right|$ converges
by Lemma \ref{lem: Brillinger}. Furthermore, $1+C\sum_{i=1}^{\infty}i\left|a_{i}\right|$
bounds every term of (\ref{eq: L1}), and hence bounds the maximum
term. Therefore, $\left\Vert \mathbf{L}_{T}\right\Vert _{1}$ is bounded,
independently of $T$.

Lastly, must bound $\left\Vert \bm{\Delta}_{T}^{-1}\right\Vert _{1}$.
Since $\bm{\Delta}_{T}^{-1}$ is a diagonal matrix and each $0<\varsigma_{k}^{2}<\infty$,
\[
\left\Vert \bm{\Delta}_{T}^{-1}\right\Vert _{1}=\max_{k\in\left[T-1\right]\cup\left\{ 0\right\} }\varsigma_{k}^{-2}
\]
is bounded since each column contains a single finite and positive
term. The proof is completed by combining the three bounds.
\end{proof}
Lemma \ref{Lem: Anderson 1} appears as Theorem 10.2.7, in \citet{Anderson1971}.
Lemma \ref{lem: Gray nonsing} is implied by Theorem 5.2 in \citet{Gray2006}.
Lemma \ref{Lem: Ibragimov-1} appears as Theorem 8 in \citet[Ch. VII]{Ibragimov1978}.
Lemma \ref{lem: billingsley-1} appears as Theorem 4.2 in \citet{Billingsley1968}.
Lemma 6 is obtained by taking the $l=1$ case (in the notation of
the source material) of \citet[Thm. 3.8.4]{Brillinger2001}. Lemma
\ref{lem: Berk} can be derived from Lemma 4 of \citet{Berk1974}.

\section*{Acknowledgements}

The author is indebted to Paul Kabaila, for his interest in this project
and for his fruitful comments. The research is funded under Australian
Research Council grants DE170101134 and DP180101192.

\bibliographystyle{apalike2}
\bibliography{MASTERBIB}

\end{document}